\documentclass[a4paper,10pt]{amsart}

\usepackage{preamble}

\title{Remarks on $\tau$-tilted versions of the second Brauer-Thrall Conjecture}

\author{Calvin Pfeifer} 

\address{ 
    Calvin Pfeifer: 
    Center for Quantum Mathematics, 
    Department of Mathematics and Computer Science, 
    University of Southern Denmark, 
    Campusvej 55, DK-5230 Odense M, Denmark 
} 

\email{capf@sdu.dk} 

\date{\today}

\begin{document}

\begin{abstract}
    In this short note, we state a stable and a $\tau$-reduced version of the second Brauer-Thrall Conjecture.
    The former is a slight strengthening of a brick version of the second Brauer-Thrall Conjecture
    raised by Mousavand and Schroll-Treffinger-Valdivieso.
    The latter is stated in terms of Geiss-Leclerc-Schröer's generically $\tau$-reduced components
    and provides a geometric interpretation of a question of Demonet.
    It follows that the stable second Brauer-Thrall Conjecture 
    implies our $\tau$-reduced second Brauer-Thrall Conjecture.
    Finally, we prove the reversed implication for the class of $E$-tame algebras
    recently introduced by Asai-Iyama.
\end{abstract}

\maketitle

\section{Introduction}

In 2014, an important new branch of Representation Theory emerged with
Adachi-Iyama-Reiten's $\tau$-tilting theory \cite{AIR14}.
It generalizes the mutation theory of quivers with potentials \cite{DWZ08}
to arbitrary finite-dimensional algebras.
This theory may be seen from two ``Koszul dual'' perspectives:
There are $\tau$-rigid modules central to $\tau$-tilting theory \cite{AIR14} on the one hand
and on the other there are bricks \cite{DIJ19} and more specifically stable modules \cite{BST19}.
In particular, there is now a remarkable theory of $\tau$-tilting finite algebras initiated in
\cite{DIJ19}.
This motivates to go beyond the finite case 
just like Brauer and Thrall went beyond representation finite algebras
a century ago.
The famous Brauer-Thrall Conjectures were an important driving force of 20th century Representation Theory 
and stimulated the development of many key techniques
to construct and classify modules over finite-dimensional algebras.
These conjectures were first published in Jan's Thesis \cite[Chapter I.1]{J54}
where he refers to unpublished notes of Brauer and Thrall
who intendet their conjectures as exercises for graduate students, 
see \cite{Rin80} for a thorough survey with historical remarks.
The first Brauer-Thrall Conjecture is now a well established theorem 
due to Roiter
and states

\begin{bti}[Theorem due to \cite{Roi68}]
    Let $A$ be a representation infinite algebra.
    For every $d\geq 0$ 
    exists an indecomposable $V\in\mod(A)$
    with $\dim(V) \geq d$.
\end{bti}

Recently, a $\tau$-tilted version of the first Brauer-Thrall Conjecture 
in terms of bricks was proved
by Schroll-Treffinger \cite{ST22} (see also \cite{MP23}).
This motivates to seek for $\tau$-tilted analogs of 
the much harder second Brauer-Thrall Conjecture
which classically states

\begin{btii}
    Let $A$ be a representation infinite algebra.
    For every $d_0\geq 0$ 
    exists $d\geq d_0$
    and infinitely many
    pairwise non-isomorphic 
    indecomposable $V\in\mod(A)$
    with $\dim(V) = d$.
\end{btii}

It is still open over general fields.
The first complete proof valid over algebraically closed fields
was achieved by Bautista \cite{B85},
see also \cite{Bon16} for a proof with remarks on its history and further references.
The second conjecture is a bit too strong to allow for direct $\tau$-tilted versions 
as the Kronecker algebra demonstrates.
Fortunately, Smal{\o} \cite{S80} provided an induction step which reduces Brauer-Thrall II 
to the a priori weaker

\begin{btii'}
    Let $A$ be a representation infinite algebra.
    There exist $d\geq 0$
    and infinitely many
    pairwise non-isomorphic 
    indecomposable $V\in\mod(A)$
    with $\dim(V) = d$.
\end{btii'}

In \cite{M22} and \cite{STV21} a $\tau$-tilted version of Brauer-Thrall II' in terms of bricks was proposed.
We decided to strengthen their conjecture slightly in terms of stable modules,
this is our Conjecture \ref{conj:stab_bt_ii} and states

\begin{stablebtii}
    Let $A$ be a $\tau$-tilting infinite algebra.
    Then there exists a dimension vector $\bd\in\K_0(A)^+$
    and a weight $\theta\in\K_0(A)^*$
    such that the dimension of the moduli space of $\theta$-stable $A$-modules
    with dimension vector $\bd$
    is strictly positive.
\end{stablebtii}

This conjecture was recently confirmed 
for the class of special biserial algebras in \cite{STV21} (see Example \ref{ex:special_biserial_algebras})
and we confirm it for the class of GLS algebras from \cite{GLS17i} in an upcoming preprint \cite{Pfe23gls} (see Example \ref{ex:gls_algebras}).
On the other hand, elementary and well known geometric considerations 
reveal that there can be only finitely isomorphism classes of $\tau$-rigid modules of a fixed dimension.
Thus, to state a $\tau$-tilted version of Brauer-Thrall II we invoke Geiß-Leclerc-Schröers $\tau$-reduced components \cite{GLS12}
as natural generalizations of $\tau$-rigid modules and arrive at our Conjecture \ref{conj:tau_bt_ii} which states

\begin{taubtii}
    Let $A$ be a $\tau$-tilting infinite algebra.
    Then there exists a dimension vector $\bd\in\K_0(A)$
    and a generically $\tau$-reduced and irreducible component 
    of $\Rep(A,\bd)$
    with generically at least one parameter.
\end{taubtii}

We use Plamondon's classification of $\tau$-reduced components \cite{Pla13}
to show that our $\tau$-reduced Brauer-Thrall II' Conjecture 
is in fact equivalent to a conjectural characterization of $\tau$-tilting finite algebras
in terms of their $\g$-vector fans due to Demonet.

\begin{demonet}\cite[Question 3.49]{D17}
    A finite-dimensional algebra $A$ is $\tau$-tilting finite
    if and only if its $\g$-vector fan is rationally complete.
\end{demonet}

With Jasso's $\tau$-tilting reduction \cite{J15}
and Brüstle-Smith-Treffingers semistability of $\tau$-perpendicular subcategories \cite{BST19},
we observe that the stable Brauer-Thrall II' Conjecture implies Demonet's Conjecture.
For the recently introduced class of $\E$-tame algebras 
from \cite{AI21} and \cite{DF15} (see Definition \ref{def:e_tame}),
we are able to prove that all three conjectures are equivalent.
This is the Main Theorem of our short note:

\begin{maintheorem}\label{thm:main}
    Let $A$ be a finite-dimensional algebra.
    Consider the following statements:
    \begin{enumerate}[label = (\roman*)]
        \item \label{enum:demonet_conjecture}
        $A$ satisfies Demonet's Conjecture.
        \item \label{enum:demonet_tau_reduced_bt_ii}
        $A$ satisfies the $\tau$-reduced Brauer-Thrall II' Conjecture.
        \item \label{enum:demonet_stable_bt_ii}
        $A$ satisfies the stable Brauer-Thrall II' Conjecture.
    \end{enumerate}
    Then the implications
    $\ref{enum:demonet_conjecture}
    \Leftrightarrow \ref{enum:demonet_tau_reduced_bt_ii}
    \Leftarrow \ref{enum:demonet_stable_bt_ii}$
    hold.
    If $A$ is $E$-tame, then 
    $\ref{enum:demonet_tau_reduced_bt_ii} \Rightarrow \ref{enum:demonet_stable_bt_ii}$
    holds as well.
\end{maintheorem}
\begin{proof}
    The equivalence $\ref{enum:demonet_conjecture}
    \Leftrightarrow \ref{enum:demonet_tau_reduced_bt_ii}$
    is Proposition \ref{prop:demonet_tau_btii}.
    The implication
    $\ref{enum:demonet_stable_bt_ii}\Rightarrow \ref{enum:demonet_conjecture}$
    is Proposition \ref{prop:demonet_stable_btii}.
    It remains to prove
    $\ref{enum:demonet_tau_reduced_bt_ii}\Rightarrow \ref{enum:demonet_conjecture}$
    for $E$-tame algebras.
    This is done in the final Section \ref{sec:main_proof}.
\end{proof}

\section{Conventions}

\subsection*{Topology}
For a topological space $\cX$ write $\dim(\cX)$ for its Krull dimension
and given $x\in \cX$ write $\dim_x(\cX)$ for the local dimension of $\cX$ at $x$.
For an irreducible topological space $\cZ$ and a property $\sfP$ of points of $\cZ$, 
we say
\begin{align*}
    \text{
    \emph{``$\cZ$ satisfies $\sfP$ generically''}
    or
    \emph{``the generic elements of $\cZ$ satisfy $\sfP$''}
    etc.
    }
\end{align*}
if there exists a non-empty open hence dense subset $\cU\subseteq \cZ$
such that every $z\in\cU$ satisfies $\sfP$.
In particular, for a constructible map $f\colon \cZ \to X$ to a set $X$
we write $f(\cZ) = x$ for some $x\in X$
if $f(z) = x$ for generic $z\in\cZ$.

\subsection*{Representations}
Throughout, we fix an algebraically closed field $K$.
All our algebras $A = KQ/I$ are finite-dimensional associative 
and given by a quiver $Q$ with admissible ideal $I\subseteq KQ$.
All our modules are finite-dimensional left $A$-modules
and we denote the category of modules by $\mod(A)$
with its Auslander-Reiten translation $\tau_A$.
We freely identify $A$-modules with $K$-linear representations of $Q$ satisfying the relations in $I$.
Let $\K_0(A)$ denote the Grothendieck group of $\mod(A)$
with standard basis given by the classes of simple $A$-modules $S_i$ indexed by vertices $i\in Q_0$.
More generally, define for any commutative ring $R$
\begin{align*}
    \K_0(A)_R := \K_0(A)\otimes_{\bZ} R, && \K_0(A)_R^* := \Hom_\bZ(\K_0(A),R).
\end{align*}
Further, let $\K_0(A)^+\subseteq \K_0(A)$ denote the submonoid of classes of $A$-modules.
Given $\bd\in\K_0(A)^+$ with $\bd = (d_i)_{i\in Q_0}$, 
write $\Rep(A,\bd)$ for the affine variety of representations of $A$ 
with dimension vector $\bd$ and 
\[
    \GL(K,\bd) := \prod_{i\in Q_0} \GL(K,d_i)
\] 
for the structure group acting on $\Rep(A,\bd)$ via conjugation.
We write $\Irr(A,\bd)$ for the set of irreducible components of $\Rep(A,\bd)$ and set
\begin{align*}
    \Irr(A) := \bigsqcup_{\bd\in\K_0(A)^+} \Irr(A,\bd).
\end{align*}
Denote the $\GL(K,\bd)$-orbit of a representation $V\in\Rep(A,\bd)$ by $\cO(V)$.
The \emph{generic number of parameters} of $\cZ\in\Irr(A)$ is
\begin{align*}
    c_A(\cZ) := \min \{\dim \cZ - \dim \cO(V) \mid V\in\cZ\}.
\end{align*}
The following functions are well-known to be upper semicontinuous
for all $\bd,\bd'\in\K_0(A)^+$:
\begin{align*}
    \dimhom_A(-,?)&\colon \Rep(A,\bd) \times \Rep(A,\bd') \to \bZ, 
    ~~~~ ~~~~ ~~~~ ~~~~ 
    (V,W) \mapsto \dim_K \Hom_A(V,W), \\
    \dimext^1_A(-,?)&\colon \Rep(A,\bd) \times \Rep(A,\bd') \to \bZ, 
    ~~~~ ~~~~ ~~~~ ~~~~ 
    (V,W) \mapsto \dim_K \Ext^1_A(V,W), \\
    \dimhom^\tau_A(-,?)&\colon \Rep(A,\bd) \times \Rep(A,\bd') \to \bZ, 
    ~~~~ ~~~~ ~~~~ ~~~~ 
    (V,W) \mapsto \dim_K \Hom_A(W,\tau_A(V)),
\end{align*}
and so are their diagonal values
\begin{align*}
    \dimend_A(-)&\colon \Rep(A,\bd) \to \bZ, 
    ~~~~ ~~~~ ~~~~ ~~~~ 
    V \mapsto \hom_A(V,V), \\
    \dimext^1_A(-)&\colon \Rep(A,\bd) \to \bZ, 
    ~~~~ ~~~~ ~~~~ ~~~~ 
    V \mapsto \dimext^1_A(V,V), \\
    \dimhom^\tau_A(-)&\colon \Rep(A,\bd) \to \bZ, 
    ~~~~ ~~~~ ~~~~ ~~~~ 
    V \mapsto \dimhom^\tau_A(V,V);
\end{align*}
see \cite{CB93} and \cite{GLFS23} for the proofs.
Therefore, their generic values $\dimhom_A(\cZ,\cZ')$, $\dimend_A(\cZ)$, etc.
coincide with their minimal values on $\cZ\times\cZ'$ respectively $\cZ$ for $\cZ,\cZ'\in\Irr(A)$.

\subsection*{Presentations}
The full exact subcategory of $\mod(A)$ of projective $A$-modules is denoted $\proj(A)$.
Let $\K^{b}(A)$ be the bounded homotopy category of complexes of projective $A$-modules 
with $\Sigma$ its shift functor.
Let $\K_0^{\proj}(A)$ be the Grothendieck group of $\proj(A)$ 
with standard basis given by the classes of indecomposable projectives $P_i$ indexed by $i\in Q_0$.
Let $\K_0^{\proj}(A)^+$ be the submonoid of $\K_0^{\proj}(A)$ consisting of classes of projective $A$-modules.
The \emph{Euler pairing} is defined for $P\in\proj(A)$ and $V\in\mod(A)$ as
\begin{align*}
    \euler{-,?}_A\colon \K_0^{\proj}(A) \times \K_0(A) \to \bZ, ~~~~
    \euler{P,V}_A := \dim_K \Hom_A(P,V).
\end{align*}
Since $A$ is by our standing assumptions a split $K$-algebra, this induces an isomorphism 
\begin{align*} \label{euler_embedding}
(-)^\vee\colon \K_0^{\proj}(A) \xrightarrow{\sim} \K_0(A)^*.
\end{align*}
Given $\bmgamma\in\K_0^{\proj}(A)^+$ written as $\bmgamma = (\gamma_i)_{i\in Q_0}$,
there is up to isomorphism a unique 
$P_\bmgamma\in\proj(A)$ with $[P_\bmgamma] = \bmgamma$ in $\K_0^{\proj}(A)$
and we consider the smooth and irreducible open subvariety 
\[
    \Proj(A,\bmgamma) := \cO(P_\bmgamma) \subseteq \Rep(A,\dimv(P_\bmgamma)).
\]
Given any $\bmgamma\in\K_0^{\proj}(A)$, set
\begin{align*}
    \Proj(A,\bmgamma) := \Proj(A,\bmgamma_1) \times \Proj(A,\bmgamma_0).
\end{align*}
where $\bmgamma = \bmgamma_0 - \bmgamma_1$ for unique $\bmgamma_0,\bmgamma_1\in\K_0^{\proj}(A)^+$
with $\bmgamma_{0} \cdot \bmgamma_{1} = 0$.
Plamondon's \emph{variety of reduced presentations} \cite[Section 2.4]{Pla13} is the variety of triples
\begin{align*}
    \Pres(A,\bmgamma) := \{\vec{P}=(p,P_1,P_0) \mid \text{$(P_1,P_0)\in\Proj(A,\bmgamma)$ and $p\in\Hom_A(P_1,P_0)$}\}
\end{align*}
upon which the group 
$\GL(K,\bmgamma) := \GL(K,\dimv(P_{\bmgamma_1})) \times \GL(K,\dimv(P_{{\bmgamma_0}}))$ 
acts via conjugation.
This is easily seen to be a $\GL(K,\bmgamma)$-equivariant vector bundle 
over the base $\Proj(A,\bmgamma)$.
In particular, $\Pres(A,\bmgamma)$ is a smooth and irreducible variety.
The following functions are well-known to be upper semicontinuous
for every $\bmgamma,\bmgamma'\in\K_0^{\proj}(A)$:
\begin{align*}
    h_A(-,?)&\colon \Pres(A,\bmgamma) \times \Pres(A,\bmgamma') \to \bZ, &&&
    (\vec{P},\vec{Q}) &\mapsto \dim_K \Hom_{K^{b}(A)}(\vec{P},\vec{Q}); \\
    e_A(-,?)&\colon \Pres(A,\bmgamma) \times \Pres(A,\bmgamma') \to \bZ, &&&
    (\vec{P},\vec{Q}) &\mapsto \dim_K \Hom_{K^{b}(A)}(\vec{P}, \Sigma(\vec{Q})); \\
    h_A(-)&\colon \Pres(A,\bmgamma) \to \bZ, &&&
    \vec{P} &\mapsto h_A(\vec{P},\vec{P}); \\
    e_A(-)&\colon \Pres(A,\bmgamma) \to \bZ, &&&
    \vec{P} &\mapsto e_A(\vec{P},\vec{P}),
\end{align*}
where we naturally consider $\vec{P} = (P_1 \xrightarrow{p} P_0)\in\K^{b}(A)$ 
with $P_0$ in degree $0$ and $P_1$ in degree $-1$.
Therefore, their generic values $h_A(\bmgamma,\bmgamma')$, $e_A(\bmgamma,\bmgamma')$, $h_A(\bmgamma), e_A(\bmgamma)$
coincide with their minimal values on 
$\Pres(A,\bmgamma)\times\Pres(A,\bmgamma')$ respectively $\Pres(A,\bmgamma)$ for $\bmgamma,\bmgamma'\in\K_0^{\proj}(A)$.

\section{Demonet's Conjecture and $\g$-vectors}

The protagonists of Adachi-Iyama-Reiten's $\tau$-tilting theory \cite{AIR14} are 
\emph{$\tau$-rigid modules}
that is modules $V\in\mod(A)$ with $\Hom_A(V,\tau_A(V)) =0$.
More generally, a pair $(V,P)\in\mod(A)\times \proj(A)$ is \emph{$\tau$-rigid}
if $V$ is $\tau$-rigid and $\Hom_A(P,V) =0$.
The \emph{$\g$-vector} of a module $V\in\mod(A)$ 
with minimal projective presentation $P_1 \to P_0 \to V \to 0$
and more generally of a pair $(V,P)$ with $P\in\proj(A)$
is defined as the class
\begin{align*}
    \g(V) := [P_0] - [P_1] \in \K_0^{\proj}(A), && \g(V, P) := \g(V) - \g(P) \in \K_0^{\proj}(A).
\end{align*}
The $\g$-vectors of indecomposable summands of $\tau$-rigid pairs $(V,P)$
span convex cones
\begin{align*}
    \sfC(V,P) := 
    \left\{\sum_{i = 1}^{m} \alpha_i \g(V_i) - \sum_{j = 1}^{m'} \beta_j \g(P_j) ~ \middle\vert ~ 
    \text{$\forall i,j:~ \alpha_{i},\beta_j \in\bR_{\geq 0}$}\right\} 
    \subseteq \K_0^{\proj}(A)_{\bR}
\end{align*}
where $V = V_1\oplus \dots\oplus V_m$ and $P = P_1\oplus \dots\oplus P_{m'}$
are their direct sum decompositions into indecomposable summands.
The \emph{$\g$-vector fan} $\fan(A)\subseteq \K_0(A)^{\proj}_\bR$ is defined as the union of all cones of $\tau$-rigid pairs.
This is a rational non-singular polyhedral fan in $\K_0^{\proj}(A)_\bR$ introduced and studied in \cite{DIJ19}.

\bigskip

It was realized by Brüstle-Smith-Treffinger \cite{BST19} that $\tau$-tilting theory
is closely related to King's theory of semistability \cite{K94}.
Given $\theta\in\K_0(A)^*_\bR$, a module $V\in\mod(A)$ is \emph{$\theta$-semistable}
if $\theta(V) = 0$ and for every proper non-zero submodule $U\subseteq V$ is $\theta(U) \leq 0$;
and $V$ is \emph{$\theta$-stable} if the inequality is always strict.
We say that a module $V\in\mod(A)$ is \emph{stable} 
if there exists a weight $\theta\in\K_0(A)^*_\bR$ 
such that $V$ is $\theta$-stable.
A bridge between stability and $\tau$-tilting theory is provided by \emph{Auslander-Reiten's $\g$-vector formula}: 
\begin{align} \label{ar_g_vector}
    \euler{\g(V,P),\dimv(X)}_A = \dimhom_A(V,X) - \dimhom_A(X,\tau_A(V)) - \dimhom_A(P,X)
\end{align}
for all $V,X\in\mod(A)$ and $P\in\proj(A)$ \cite[Theorem 1.4]{AR85}.
Indeed, this formula allows us to identify parts of King's semistable subcategories
\begin{align*}
    \cW(\theta) := \{V\in\mod(A) \mid \text{$V$ is $\theta$-semistable}\}, ~~~~
    \text{for $\theta\in\K_0(A)^*$}
\end{align*}
as Jasso's \emph{$\tau$-perpendicular subcategories} 
\cite[Definition 3.3]{J15}
\begin{align*}
    \cW(V,P) := V^\perp \cap {^\perp(\tau_A(V))} \cap P^\perp, ~~~~
    \text{for $(V,P)\in\mod(A)\times\proj(A)$}.
\end{align*}
Here we write $V^\perp := \{X\in\mod(A) \mid \Hom_A(V,X) = 0\}$ and 
${^\perp V} := \{X\in\mod(A) \mid \Hom_A(X,V) = 0\}$ for $V\in\mod(A)$.

\begin{lemma}\label{lem:tau_perp_serre}
    Let $\theta\in\K_0(A)^*$ and $(V,P)\in\mod(A)\times\proj(A)$
    with $\g(V,P)^\vee = \theta$. Then
    \begin{align*}
        \cW(V,P) \subseteq \cW(\theta)
    \end{align*}
    is a Serre subcategory of $\cW(\theta)$.
\end{lemma}
\begin{proof}
    Let $W\in\cW(V,P)$ and $U\subseteq W$, then by (\ref{ar_g_vector})
    \begin{align*}
        \theta(U) = \euler{\g(V,P),\dimv(U)}_A = \dimhom_A(V,U) - \dimhom_A(U,\tau_A(V)) - \dimhom_A(P,U) \leq 0
    \end{align*}
    because $\hom_A(V,U) = \hom_A(P,U)= 0$ 
    for $W\in V^\perp\cap P^{\perp}$ and $U$ is a submodule of $W$.
    Also, $\theta(W) = 0$ hence $W\in\cW(\theta)$.
    If moreover $U\in\cW(\theta)$,
    then $\theta(U) = 0$ 
    hence we must have $\hom_A(U,\tau_A(V)) = 0$.
    Therefore $U\in\cW(V,P)$.
    Similarly, every $\theta$-semistable factor of $W$ lies in $\cW(V,P)$.
    Finally, it is clear that $\cW(V,P) := V^\perp \cap {^\perp(\tau_A V)} \cap P^\perp$ is closed under extensions.
\end{proof}

This lemma can be found in \cite[Lemma 3.8]{AI21} in terms of presentations;
since the proof is straight forward and our proof of the Main Theorem \ref{thm:main} heavily relies on it, 
we include it here for convenience.
The proof strategy is the same as in \cite[Section 3.3]{BST19} 
where they prove equality in Lemma \ref{lem:tau_perp_serre}
if $\theta$ lies in the relative interior of $\sfC(V,P)^\vee$ for a $\tau$-rigid pair $(V,P)$. 

The outlined correspondences are particularly appealing in the $\tau$-tilting finite case:

\begin{theorem}\cite{DIJ19}\cite{BST19}\cite{Asa21} \label{thm:tau_finite}
    For a finite-dimensional algebra $A$, the following statements are equivalent
    \begin{enumerate}[label = (\roman*)]
        \item There are only finitely many isomorphism classes of indecomposable $\tau$-rigid $A$-modules;
        \item There are only finitely many isomorphism classes of stable $A$-modules;
        \item The $\g$-vector fan of $A$ is complete i.e. $\fan(A) = \K_0^{\proj}(A)_\bR$.
    \end{enumerate}
\end{theorem}

An algebra $A$ is said to be \emph{$\tau$-tilting finite} 
if $A$ satisfies any of the equivalent properties in Theorem \ref{thm:tau_finite}.
Of course, an algebra $A$ is \emph{$\tau$-tilting infinite} if $A$ is not $\tau$-tilting finite.
Demonet asks in \cite[Question 3.49]{D17} whether it suffices for $A$ to be $\tau$-tilting finite
that $\fan(A)$ is ``rationally complete''
in the following sense:

\begin{conjecture}[Demonet's Conjecture] \label{conj:demonet}
    A finite-dimensional algebra $A$
    is $\tau$-tilting finite
    if and only if 
    $\K_0^{\proj}(A)_\bZ \subseteq \fan(A)$.
\end{conjecture}

\section{The $\tau$-reduced Brauer-Thrall Conjectures}

To illustrate the need of $\tau$-reduced components 
in the formulation of a $\tau$-tilted version of the second Brauer-Thrall Conjecture,
we start with an obvious $\tau$-tilted version of the first Brauer-Thrall Conjecture.

\begin{proposition}\label{prop:tau_bt_i}
    Let $A$ be a $\tau$-tilting infinite algebra.
    Then,
    for every $d\geq 0$
    there exists an indecomposable 
    $\tau$-rigid $V\in\mod(A)$
    with $\dim(V) \geq d$.
\end{proposition}
\begin{proof}
    Let $\bd\in \K_0(A)^+$ be a fixed dimension vector
    and consider the variety of representations $\Rep(A,\bd)$.
    The orbit closure $\overline{\cO(V)}$ of any $\tau$-rigid module $V$ in $\Rep(A,\bd)$ 
    is an irreducible component of $\Rep(A,\bd)$ by Voigt's Isomorphism \cite[Section 1.1]{Gab74}.
    But $\Rep(A,\bd)$ has only finitely many irreducible components as an affine variety.
    This proves that there are only finitely many isomorphism classes 
    of $\tau$-rigid $A$-modules $V$
    with fixed dimension vector $\dimv(V) = \bd$.
    By the pigeonhole principle,
    there must exist indecomposable $\tau$-rigid $A$ modules
    of arbitrarily large dimension
    if $A$ is $\tau$-tilting infinite.
\end{proof}

The proof shows that there cannot be infinitely many pairwise non-isomorphic
$\tau$-rigid modules of a fixed dimension.
Therefore, we need to enlarge our class of $\tau$-rigid modules.
A natural geometric generalization is provided by Geiß-Leclerc-Schröer's
generically $\tau$-reduced components.
Note that Voigt's Isomorphism \cite[Section 1.1]{Gab74} and 
Auslander-Reiten Duality \cite{AR75} yield the inequalities
\begin{align*}
    c_A(\cZ) \leq \dimext^1_A(\cZ) \leq \dimhom^\tau_A(\cZ).
\end{align*}
A component $\cZ\in\Irr(A)$ is \emph{generically $\tau$-reduced} if
\begin{align*}
    c_A(\cZ) = \dimhom_A^\tau(\cZ,\cZ).
\end{align*}
Write $\Irr^\tau(A)$ for the set of generically $\tau$-reduced components in $\Irr(A)$.
They first appeared for certain Jacobi algebras in \cite[Section 1.5]{GLS12} 
under the name \emph{strongly reduced components}
and were defined and studied in full generality in \cite{Pla13}.
Indeed, generically $\tau$-reduced components are generically reduced.
Further, there is a bijective correspondence
\begin{align*}
    \{\text{Isomorphism classes of $\tau$-rigid $A$-modules $V$}\}
    \xrightarrow{1:1}
    \{\text{$\cZ\in\Irr^\tau(A)$ with $c_A(\cZ) =0$}\}
\end{align*}
given by sending a module $V$ to the closure of its orbit $\overline{\cO(V)}$.
We refer to 
\cite[Section 2]{Pla13}, 
\cite[Section 5]{CILFS15}, 
\cite[Sections 2--4]{GLFS22} and 
\cite{GLFS23} 
for further background.
It is now natural to come up with the following $\tau$-reduced version of the second Brauer-Thrall Conjecture.

\begin{conjecture}[$\tau$-reduced Brauer-Thrall II'] \label{conj:tau_bt_ii}
    Let $A$ be a $\tau$-tilting infinite algebra.
    Then there exist $\bd\in\K_0(A)^+$
    and a generically indecomposable $\cZ\in\Irr^\tau(A,\bd)$
    with $c_A(\cZ) \geq 1$.
\end{conjecture}

\begin{example}
    Let $A = KQ$ be the path algebra of the Kronecker quiver 
    $
    \begin{tikzcd} [sep = .5cm]
        Q\colon 2 \ar[r, shift left = .75mm]\ar[r, shift right = .75mm] & 1
    \end{tikzcd}
    $.
    Then $\Rep(A,\bd)$ is irreducible and smooth for every $\bd\in\K_0(A)^+$.
    Since $A$ is hereditary, it follows that $\Rep(A,\bd)$ is generically $\tau$-reduced.
    It is well known that $\Rep(A,\bd)$ is generically indecomposable 
    with $c_A(\Rep(A,\bd)) \geq 1$ if and only if $\bd = (1,1)$.
    Therefore, $A$ satisfies our $\tau$-reduced Brauer-Thrall II' Conjecture \ref{conj:tau_bt_ii}
    and provides a counterexample for a $\tau$-reduced version of Brauer-Thrall II.
    More generally, every representation infinite path algebra satisfies Conjecture \ref{conj:tau_bt_ii}.
\end{example}

Note that taking $\g$-vectors defines constructible maps 
$\g\colon \Rep(A,\bd) \to \K_0^{\proj}(A)$
for every $\bd\in\K_0(A)^+$
hence any $\cZ\in\Irr(A)$ possesses a \emph{generic $\g$-vector} $\g(\cZ)\in\K_0^{\proj}(A)$.
Plamondon proves that $\tau$-reduced components are determined by their \emph{generic $\g$-vector}:
Let $\bmgamma\in\K_0^{\proj}(A)$.
Building upon Palu's constructibility of cokernels
\cite[Lemma 2.3]{Pal11},
Plamondon constructs in 
\cite[Lemma 2.11]{Pla13} 
a regular map on an open dense subset $\cU\subseteq\Pres(A,\bmgamma)$
\begin{align*}
    \Phi\colon \cU \to \Rep(A,\bd)
\end{align*}
for some $\bd\in\K_0(A)^+$ 
such that $\Phi(\vec{P}) \cong \Cok(p)$ for all $\vec{P} = (p,P_1,P_0)\in\cU$.
He then considers the irreducible and closed subset 
\[
    \cP(\gamma) := \overline{\Phi(\cU)}\subseteq \Rep(A,\bd).
\]

\begin{theorem}\cite[Theorem 1.2]{Pla13}\label{thm:plamondon}
    The maps
    \[
        \begin{tikzcd}
            \Irr^\tau(A)
            \ar[rr, shift left = 1mm, "\g"]
            &&
            \ar[ll, shift left = 1mm, "\cP"]
            \K_0^{\proj}(A)
        \end{tikzcd}
    \]
    satisfy 
    \begin{enumerate}[label = (\roman*)]
        \item for all $\cZ\in\Irr^\tau(A)$ is $\cP(\g(\cZ)) = \cZ$;
        \item for all $\bmgamma\in\K_0^{\proj}(A)$ is $-\bm{\delta}:=\g(\cP(\bmgamma))-\bmgamma\in \K_0^{\proj}(A)$ with $e_A(\g(\cP(\bmgamma)),\bm{\delta}) = e_A(\bm{\delta},\g(\cP(\bmgamma))) = 0$.
    \end{enumerate}
\end{theorem}

The bridge between presentations and $\tau$-tilting theory is provided by a precursor of the Auslander-Reiten formula:

\begin{lemma} \cite[Lemma 2.6]{Pla13} \label{pre_ar}
    Let $V,W\in \mod(A)$ 
    with minimal projective presentations $\vec{P}(V)$ and $\vec{P}(W)$.
    Then
    \begin{align*}
        \Hom_{K^b(A)}(\vec{P}(V),\Sigma(\vec{P}(W))) \cong \D\Hom_A(W,\tau_A(V)).
    \end{align*}
\end{lemma}

\begin{lemma} \label{lem:tau_hom_e_invariant}
    Let $\cZ,\cZ'\in\Irr^\tau(A)$. Then
    \begin{align*}
        \dimhom^\tau_A(\cZ,\cZ') = e_A(\g(\cZ),\g(\cZ')) && \dimhom^\tau_A(\cZ) = e_A(\g(\cZ)). 
    \end{align*}
\end{lemma}
\begin{proof}
    Set $\bmgamma := \g(\cZ)$ and $\bmgamma' := \g(\cZ')$.
    Let $\cV\subseteq \cZ$ and $\cV'\subseteq\cZ'$ be open dense subsets 
    with $\dimhom^\tau_A(V,V') = \dimhom^\tau_A(\cZ,\cZ')$ 
    for all $V\in\cV$ and $V'\in\cV'$.
    Let $\cU\subseteq \Pres(A,\bmgamma)$ and $\cU'\subseteq\Pres(A,\bmgamma')$ be open dense subsets
    with $e_A(\vec{P},\vec{P}') = e_A(\bmgamma,\bmgamma')$
    for all $\vec{P}\in\Pres(A,\bmgamma)$ and $\vec{P}'\in\Pres(A,\bmgamma')$.
    For $\vec{P}\in\cU\cap\Phi^{-1}(\cV)$ and $\vec{P'}\in\cU'\cap\Phi^{-1}(\cV')$
    we find
    \begin{align*}
        e_A(\bmgamma,\bmgamma') = e_A(\vec{P},\vec{P'}) = \dimhom^\tau_A(V,V') = \dimhom^\tau_A(\cZ,\cZ').
    \end{align*}
    where the second equality holds by Lemma \ref{pre_ar}.
    Similarly, one shows $e_A(\bmgamma) = \dimhom^\tau_A(\cZ)$.
\end{proof}

\begin{lemma} \label{lem:tau_reduced_fan}
    For $\bmgamma\in\K_0^{\proj}(A)$, the following are equivalent
    \begin{enumerate}[label = (\roman*)]
        \item \label{enum:fan}
        $\bmgamma\in\fan(A)$;
        \item \label{enum:e_vanishes}
        $e_A(\bmgamma) = 0$;
        \item \label{enum:c_vanishes}
        $c_A(\cP(\bmgamma)) = 0$.
    \end{enumerate}
\end{lemma}
\begin{proof}
    By definition of the $\g$-vector fan,
    we have $\bmgamma\in\fan(A)$ if and only if 
    there exists a $\tau$-rigid pair $(V,P)$
    with $\g(V,P) = \bmgamma$.

    $\ref{enum:fan}\Rightarrow\ref{enum:e_vanishes}, \ref{enum:c_vanishes}$:
    Let $(V,P)$ be a $\tau$-rigid pair with $\g(V,P) = \bmgamma$.
    Take a minimal projective presentation $\vec{P}'$ of $V$
    and consider $\vec{P} := \vec{P}' \oplus \Sigma(P)\in\Pres(A,\bmgamma)$.
    Then $e_A(\vec{P}) = e_A(\vec{P}') = \dimhom^\tau_A(V) = 0$.
    This shows $e_A(\bmgamma) = 0$.
    Moreover, $\overline{\cO(V)}$ is a $\tau$-reduced component
    with $\g(V) = \g(\cP(\bmgamma))$, hence $\cP(\bmgamma) = \overline{\cO(V)}$
    by Theorem \ref{thm:plamondon}.
    This shows $c_A(\cP(\bmgamma)) = 0$.

    $\ref{enum:e_vanishes}\Rightarrow \ref{enum:fan}$:
    Let $\vec{P}\in\Pres(A,\bmgamma)$ with $e_A(\vec{P}) = 0$.
    Write $\vec{P} = \vec{P}' \oplus \Sigma(P)$ for some $P\in\proj(A)$ 
    and $\vec{P}$ a minimal projective presentation of some $V\in\mod(A)$.
    Then $\dimhom^\tau_A(V,V) = e_A(\vec{P}',\vec{P}') \leq e_A(\vec{P},\vec{P}) = 0$
    hence $(V,P)$ is a $\tau$-rigid pair with $\g(V,P) = \g(\vec{P}) = \bmgamma$.

    $\ref{enum:c_vanishes}\Rightarrow \ref{enum:e_vanishes}$:
    Let $\bmgamma\in\K_0^{\proj}(A)$ 
    and set $\bmgamma' := \g(\cP(\bmgamma))$.
    Write $\bmgamma = \bmgamma' + \bm{\delta}$
    as in Plamondon's Theorem \ref{thm:plamondon}.
    With Lemma \ref{lem:tau_hom_e_invariant} we find:
    \[
        e_A(\bmgamma) = e_A(\bmgamma') + e_A(\bm{\delta}) + e_A(\bmgamma',\bm{\delta}) + e_A(\bm{\delta},\bmgamma') = e_A(\bmgamma') = c_A(\cP(\bmgamma)) = 0
    \]
\end{proof}

\begin{proposition}\label{prop:demonet_tau_btii}
    Let $A$ be a finite-dimensional algebra.
    The following statements are equivalent
    \begin{enumerate}[label = (\roman*)]
        \item \label{enum:prop1_demonet}
        $A$ satisfies Demonet's Conjecture \ref{conj:demonet};
        \item \label{enum:prop1_tau_btii}
        $A$ satisfies the $\tau$-reduced Brauer-Thrall II' Conjecture \ref{conj:tau_bt_ii}.
    \end{enumerate}
\end{proposition}
\begin{proof}
    This is an immediate consequence of 
    Lemmas \ref{lem:tau_reduced_fan}, \ref{lem:tau_hom_e_invariant} and 
    Plamondon's Theorem \ref{thm:plamondon}.
\end{proof}

\section{The stable Brauer-Thrall Conjectures}

With similar geometric arguments as in the proof of Proposition \ref{prop:tau_bt_i},
Mousavand-Paquette \cite[Theorem 6.2]{MP23} prove a stable analog of the first Brauer-Thrall Conjecture.
Another proof using $\bc$- and $\bg$-vectors, 
thus valid over not necessarily algebraically closed fields, 
is due to Schroll-Treffinger \cite[Theorem 1.1]{ST22}. 
Both work with bricks instead of stables, 
but their bricks are well known to be stable by \cite[Proposition 3.13]{BST19}.

\begin{proposition}[Stable Brauer-Thrall I]\label{prop:stable_bt_i}
    Let $A$ be a $\tau$-tilting infinite algebra.
    Then, for every $d\geq 0$
    there exists a stable $V\in\mod(A)$
    with $\dim(V)\geq d$.
\end{proposition}

A brick version of Brauer-Thrall II' was formulated in 
\cite[Conjecture 1.3.(2)]{M22} and \cite[Conjecture 2]{STV21}.
We strengthen their conjecture slightly to a stable version.
Given a dimension vector $\bd\in\K_0(A)^+$ 
and a weight $\theta\in\K_0(A)^*$,
denote by $\Rep(A,\bd,\theta)^\st$
the open subvariety of $\Rep(A,\bd)$
consisting of $\theta$-stable representations.
King \cite{K94} constructs a fine moduli space 
of $\theta$-stable $A$-modules in dimension $\bd$.
This is a quasi-projective variety
$\Mod(A,\bd,\theta)^\st$ together with a surjective map of varieties
\begin{align*}
    \pi\colon \Rep(A,\bd,\theta)^\st \twoheadrightarrow \Mod(A,\bd,\theta)^\st
\end{align*}
such that $\pi(V) = \pi(W)$ if and only if $V \cong W$.

\begin{conjecture}[Stable Brauer-Thrall II'] \label{conj:stab_bt_ii}
    Let $A$ be a $\tau$-tilting infinite algebra.
    Then there exist $\bd\in K_0(A)^+$
    and $\theta \in \K_0(A)^*$
    such that $\dim \Mod(A,\bd,\theta)^\st \geq 1$.
\end{conjecture}

\begin{example} \label{ex:path_algebras}
    Let $A = KQ$ be a representation infinite path algebra.
    Then there exists an idempotent $e\in A$ such that $A/\ideal{e} \cong KQ'$
    for a quiver $Q'$ of affine type. Thus, we may assume that $Q$ is of affine type.
    The homogeneous quasi-simple regular $A$-modules are a 1-parameter family of stables 
    with respect to the defect.
    Therefore, $A$ satisfies Conjecture \ref{conj:stab_bt_ii}.
\end{example}

\begin{example} \label{ex:gls_algebras}
    Generalizing Example \ref{ex:path_algebras},
    let $A = H(C,D,\Omega)$ be a GLS algebra \cite{GLS17i}
    associated to a generalized Cartan matrix $C$ with symmetrizer $D$ and orientation $\Omega$.
    Geiß-Leclerc-Schröer prove in \cite{GLS20}
    that $H$ is $\tau$-tilting finite if and only if $C$ is of finite type.
    Assume that $C$ is of affine type.
    In an upcoming preprint \cite{Pfe23gls}
    we provide a generic classification of locally free representations of $H$.
    In the course of our proof of this generic classification,
    we construct a 1-parameter family of stable representations of $H$.
    In particular, GLS algebras satisfy Conjecture \ref{conj:stab_bt_ii}.
\end{example}

\begin{example} \label{ex:special_biserial_algebras}
    Let $A$ be a $\tau$-tilting infinite special biserial algebra.
    It is proved in \cite[Theorem 1.1]{STV21} 
    that there exists a family of pairwise non-isomorphic
    band modules $V_\lambda$ for $\lambda\in K^\times$
    which are all bricks.
    But as band modules they satisfy 
    $\tau_A(V_\lambda)\cong V_\lambda$ 
    by the discussion in \cite[pp.165--166]{BR87}.
    Thus $V_\lambda$ must be stable 
    with respect to the weight $\g(V_\lambda)^\vee$
    by Auslander-Reiten's $\g$-vector formula (\ref{ar_g_vector}).
\end{example}

\begin{proposition}\label{prop:demonet_stable_btii}
    Let $A$ be a finite-dimensional algebra.
    Consider the following statements
    \begin{enumerate}[label = (\roman*)]
        \item \label{enum:prop2_stable_btii}
        $A$ satisfies the stable Brauer-Thrall II' Conjecture \ref{conj:stab_bt_ii}.
        \item \label{enum:prop2_demonet}
        $A$ satisfies Demonet's Conjecture \ref{conj:demonet}.
    \end{enumerate}
    Then the implication 
    $\ref{enum:prop2_stable_btii}\Rightarrow\ref{enum:prop2_demonet}$
    holds.
\end{proposition}
\begin{proof}
    Let $\bmgamma\in\fan(A)\cap\K_0^{\proj}(A)$ and $\theta := \bmgamma^\vee$
    We find $\bmgamma = \g(V,P)$ for a $\tau$-rigid pair $(V,P)$.
    By \cite[Proposition 3.13]{BST19} we have $\cW(\theta) = \cW(V,P)$.
    By \cite[Theorem 3.8]{J15} we have $\cW(V,P)\simeq\mod(B)$
    for a finite-dimensional algebra $B$.
    In particular, there are only finitely many simple objects in $\cW(\theta)$.
    But $\theta$-stable modules are precisely simple objects in $\cW(\theta)$.
    Therefore, there are only finitely many $\theta$-stable modules.
    In other words $\dim\Mod(A,\bd,\theta)^\st = 0$ for all $\bd\in\K_0(A)^+$.
\end{proof}

\section{$E$-tame algebras}

Recently, different notions of ``tameness'' in $\tau$-tilting theory gained some interest,
see e.g. \cite[Section 3.3]{BST19}, \cite{AY23}, \cite{PY23} and \cite[Section 6]{AI21}.
They also play a role in understanding the $\tau$-tilted Brauer-Thrall II Conjectures.
It is for example enough to prove Demonet's Conjecture 
for the class of \emph{$\g$-tame algebras},
those algebras $A$ whose $\g$-vector fan $\fan(A)$ is dense in $\K_0^{\proj}(A)_\bR$.
On the other hand, 
the main result of this note is that all our $\tau$-tilted versions of Brauer-Thrall II
are equivalent for the class of $\E$-tame algebras.

\begin{definition} \label{def:e_tame} \cite[Definition 6.3]{AI21} \cite[Definition 4.6]{DF15}
    A finite-dimensional algebra $A$ is \emph{$E$-tame} if $e_A(\bmgamma,\bmgamma) = 0$ 
    for all $\bmgamma\in\K_0^{\proj}(A)$.
\end{definition}

First, note that it
is enough to check $E$-tameness 
for generically indecomposable $\tau$-reduced components;
we include a proof for convenience:

\begin{lemma}\label{lem:e_tau_tame_equivalence}
    The following are equivalent
    for a finite-dimensional algebra $A$:
    \begin{enumerate}[label = (\roman*)]
        \item \label{enum:e_tame_equivalence}
        $A$ is $E$-tame.
        \item \label{enum:e_tame_indecomposable}
        Every $\bmgamma\in\K_0^{\proj}(A)$ with generically indecomposable $\Pres(A,\bmgamma)$ 
        satisfies $e_A(\bmgamma,\bmgamma) = 0$.
        \item \label{enum:tau_reduced_tame_equivalence}
        Every $\cZ\in\Irr^\tau(A)$ satisfies $\dimhom^\tau_A(\cZ,\cZ) = 0$.
        \item \label{enum:tau_reduced_tame_indecomposable}
        Every generically indecomposable $\cZ\in\Irr^\tau(A)$ satisfies $\dimhom^\tau_A(\cZ,\cZ) = 0$.
    \end{enumerate}
\end{lemma}
\begin{proof}
    The implications 
    $\ref{enum:e_tame_equivalence} \Rightarrow \ref{enum:e_tame_indecomposable}$
    and $\ref{enum:tau_reduced_tame_equivalence} \Rightarrow \ref{enum:tau_reduced_tame_indecomposable}$
    are obvious. 
    The reversed implications 
    $\ref{enum:e_tame_indecomposable} \Rightarrow \ref{enum:e_tame_equivalence}$
    and $\ref{enum:tau_reduced_tame_indecomposable} \Rightarrow \ref{enum:tau_reduced_tame_equivalence}$ 
    follow immediately from the corresponding generic Krull-Remak-Schmidt decompositions, see
    \cite[Theorem 4.4]{DF15}\cite[Theorem 2.7]{Pla13} for presentations
    and \cite[Theorem 5.11]{CILFS15} for generically $\tau$-reduced components.

    $\ref{enum:e_tame_equivalence} \Rightarrow \ref{enum:tau_reduced_tame_equivalence}$:
    Let $\cZ\in\Irr^\tau(A)$ with generic $\g$-vector $\bmgamma := \g(\cZ)$. 
    We have seen in Lemma \ref{lem:tau_hom_e_invariant} that $\dimhom^\tau_A(\cZ,\cZ) = e_A(\bmgamma,\bmgamma)$
    and $e_A(\bmgamma,\bmgamma) = 0$ because $A$ is assumed to be $E$-tame.

    $\ref{enum:tau_reduced_tame_equivalence} \Rightarrow \ref{enum:e_tame_indecomposable}$:
    Let $\bmgamma\in\K_0^{\proj}(A)$ be generically indecomposable.
    Then $\bmgamma = \g(\cZ)$ for a generically $\tau$-reduced component
    or $\bmgamma = -\g(P)$ for a projective $P\in\proj(A)$
    by Plamondon's Theorem \ref{thm:plamondon}.
    In the former case we have $e_A(\bmgamma,\bmgamma) = \dimhom^\tau_A(\cZ,\cZ) = 0$ 
    by Lemma \ref{lem:tau_hom_e_invariant} and assumption.
    In the latter case we obviously have $e_A(\bmgamma,\bmgamma) = 0$.
\end{proof}

To motivate the concept of $E$-tameness, 
we introduce a more natural notion of generically $\tau$-reduced tameness.
Note that for a representation tame algebra $A$ 
the generic number of parameters 
of a generically indecomposable $\cZ\in\Irr(A)$
satisfies the bound $c_A(\cZ)\leq 1$.
Hence, we arrive at the following definition.

\begin{definition}
    A finite-dimensional algebra $A$ is \emph{generically $\tau$-reduced tame}
    if $c_A(\cZ)\leq 1$ for all generically indecomposable $\cZ\in\Irr^\tau(A)$.
\end{definition}

Of course, there is a corresponding reformulation 
in terms of varieties of reduced presentations
but we leave this to the reader.
The relation with representation and $E$-tameness is a consequence of well-known facts:

\begin{proposition} \label{prop:tame}
    Let $A$ be a finite-dimensional algebra.
    Consider the following statements:
    \begin{enumerate}[label = (\roman*)]
        \item \label{enum:rep_tame} $A$ is representation tame.
        \item \label{enum:gen_tau_tame} $A$ is generically $\tau$-reduced tame.
        \item \label{enum:e_tame} $A$ is $E$-tame.
    \end{enumerate}
    Then the implications 
    $\ref{enum:rep_tame} \Rightarrow \ref{enum:gen_tau_tame} \Rightarrow \ref{enum:e_tame}$
    hold.
\end{proposition}
\begin{proof}
    $\ref{enum:rep_tame} \Rightarrow \ref{enum:gen_tau_tame}$:
    We already noted that $c_A(\cZ)\leq 1$ for generically indecomposable $\cZ\in\Irr(A)$
    (see \cite[Lemma 3]{CC15inv} for a proof).
    In particular, $A$ is generically $\tau$-reduced tame.

    $\ref{enum:gen_tau_tame}\Rightarrow\ref{enum:e_tame}$:
    Let $\cZ\in\Irr^\tau(A)$, we may assume that $\cZ$ is generically indecomposable by Lemma \ref{lem:e_tau_tame_equivalence}.
    Nothing is to show if $0 = c_A(\cZ) = \hom^\tau_A(\cZ)$.
    If $c_A(\cZ) = 1$, then $\hom^\tau_A(\cZ,\cZ) < \hom^\tau_A(\cZ) = c_A(\cZ)$ by
    \cite[Theorem 1.5]{GLFS23}. Thus, $\hom^\tau_A(\cZ,\cZ) = 0$.
\end{proof}

Let us next provide two examples.
The former shows that the implication 
$\ref{enum:e_tame} \Rightarrow \ref{enum:gen_tau_tame}$
in the previous proposition
does not hold.
The latter is a $\tau$-tilting infinite 
counterexample for the implication 
$\ref{enum:gen_tau_tame} \Rightarrow \ref{enum:rep_tame}$.

\begin{example}
    Let $n\geq 1$ and $H_n = KQ/I_n$ the algebra given by
    \[
            \begin{tikzcd}
                Q \colon &
                2 \ar[loop, in = 150, out = -150, distance = 5ex, "d"] 
                \ar[r, shift left = 1mm, "a"] \ar[r, shift right = 1mm, "b"'] &
                1 \ar[loop, in = 30, out = -30, distance = 5ex, "c"']
                &&&&
                I_n = \ideal{c^n,d^n, ca - ad, cb - bd}.
            \end{tikzcd}
    \]
    This is the Kronecker algebra for $n = 1$, 
    thus representation tame hence generically $\tau$-reduced tame
    and $E$-tame by Proposition \ref{prop:tame}.
    For general $n\geq 1$, this is the GLS algebra $H_n = H(C,nD,\Omega)$ from
    \cite{GLS17i} associated to
    \begin{align*}
        C = \rsm{ 2 & -2 \\ -2 & 2}, && nD = \rsm{n & 0 \\ 0 & n}, && \Omega = \{(1,2)\}. 
    \end{align*}
    Note that $c+d$ lies in the center of $H_n$ and $H_n/\ideal{c+d} \cong H_1$.
    By Eisele-Janssens-Raedschelders' reduction \cite[Equivalence (4.2)]{EJR18}
    $H_n$ is $\E$-tame since $H_1$ is $\E$-tame.
    Consider the unique irreducible component $\cZ\in\Irr(H_n, (n,n))$
    whose generic element is \emph{locally free},
    that means $c$ and $d$ act on generic $V\in\cZ$ by linear operators of maximal rank $n-1$; 
    see \cite[Proposition 3.1]{GLS18ii}.
    For $\lambda\in K^n$ consider the representation
    \[
        \begin{tikzcd}[ampersand replacement=\&]
            V_{\lambda}\colon 
            \& \&
            K^n \ar[loop, out=150, in=-150, distance=5ex, "\text{$N_n$}"']
            \ar[r, shift left = 1mm, "\text{$\bid_n$}"] 
            \ar[r, shift right = 1mm, "\text{$L_\lambda$}"']
            \&
            K^n \ar[loop, out=30, in=-30, distance=5ex, "\text{$N_n$}"]
        \end{tikzcd}
    \]
    where $\bid_n$ is the identity matrix of size $n$,
    $N_n$ is the lower triangular nilpotent Jordan block of size $n$
    and $L_\lambda$ is the lower triangular matrix with entries
    \begin{align*}
        (L_\lambda)_{ij}
        :=
        \begin{cases}
            \lambda_{i-j+1} & \text{if $i \geq j$} \\
            0 & \text{else.}
        \end{cases}
        &&
        L_\lambda =
        \left(
        \begin{matrix}
            \lambda_1 & 0 & \cdots & 0 \\
            \lambda_2 & \lambda_1 & \cdots & 0 \\
            \vdots & \vdots & \ddots & 0 \\
            \lambda_n & \lambda_{n-1} & \cdots & \lambda_1
        \end{matrix}
        \right).
    \end{align*}
    Then $V_\lambda\in\cZ$ and they satisfy
    \begin{align*}
        \dimhom_{H_n}(V_\lambda,V_\mu)
        =
        \begin{cases}
            0 & \text{if $\lambda_1\neq\mu_1$,} \\
            n & \text{if $\lambda = \mu$,}
        \end{cases}
    \end{align*}
    for all $\lambda,\mu\in K^n$.
    By comparing dimensions one finds
    \begin{align*}
        \cZ = \overline{\bigcup_{\lambda\in K^n} \cO(V_\lambda)}.
    \end{align*}
    Moreover, the generic element of $\cZ$ has projective dimension at most 1 
    by \cite[Theorem 1.2]{GLS17i}. 
    Therefore,
    \[
        c_{H_n}(\cZ) = \dimhom^\tau_{H_n}(\cZ) = \dimext^1_{H_n}(\cZ) = \dimend_{H_n}(\cZ) = n,
    \]
    the first equality holds because $\cZ$ is generically $\tau$-reduced,
    the second follows from Auslander-Reiten duality,
    and the third holds by \cite[Proposition 4.1]{GLS17i}. 
\end{example}

\begin{example}
    Consider the algebra $H = KQ/I$ given by
    \[
    \begin{tikzcd}
        Q\colon ~
        2 \ar[r, "\alpha"] &
        1 \ar[loop, out=30, in=-30, distance=5ex, "\varepsilon"] &&&&
        I := \ideal{\varepsilon^4}.
    \end{tikzcd}
    \]
    This is the GLS algebra $H = H(C,D,\Omega)$ \cite{GLS17i} associated to
    \begin{align*}
        C = \rsm{ 2 & -1 \\ -4 & 2}, && D = \rsm{4 & 0 \\ 0 & 1}, && \Omega = \{(1,2)\}. 
    \end{align*}
    The generalized Cartan matrix $C$ is of affine type.
    We show in \cite{Pfe23gls}
    that $H$ is generically $\tau$-reduced tame.
    On the other hand, its Galois-covering has a convex hypercritical subcategory
    of double extended type $\widetilde{\widetilde{\D}}_5$
    thus $H$ is representation wild.
\end{example}

\section{Proof of the Main Theorem} \label{sec:main_proof}

In this final section we finish our proof of the Main Theorem \ref{thm:main} of the present short note.
It remains to consider $E$-tame algebras
thus assume from now on that $A$ is $E$-tame and $\tau$-tilting infinite.
By \ref{enum:demonet_tau_reduced_bt_ii} 
there exists $\cZ\in\Irr^\tau(A)$ with $c_A(\cZ)\geq 1$.
Since $A$ is assumed to be $E$-tame, 
we must have $\dimhom^\tau_A(\cZ,\cZ) = 0$ according to Lemma \ref{lem:e_tau_tame_equivalence}.
Our final Lemma \ref{lem:stable_family_from_e_tame} below constructs infinitely many pairwise non-isomorphic stable $A$-modules
as factors of generic elements of $\cZ$.
This then finishes our proof of Theorem \ref{thm:main}.

\begin{lemma} \label{lem:stable_family_from_e_tame}
    Let $\cZ\in\Irr^\tau(A)$ with $\dimhom^\tau_A(\cZ,\cZ) = 0$ and $c_A(\cZ) \geq 1$.
    Set $\theta := \g(\cZ)^\vee \in \K_0(A)^*$.
    Then there exists $\bd\in\K_0(A)^+$ such that
    $\dim {\Mod(A,\bd,\theta)^\st} \geq 1$.
\end{lemma}
\begin{proof}
    Let $\cU'\subseteq \cZ \times \cZ$ be the non-empty open subset such that
    \begin{align*}
        0 = \Hom_A(W,\tau_A(V)),
    \end{align*}
    and $\g(V) = \g(W) = \g(\cZ)$
    for all $(V,W)\in\cU'$.
    Let $p_1,p_2\colon \cZ \times \cZ \to \cZ$ denote the canonical projections
    onto the first respectively second factor.
    Since projections are open, 
    $p_1(\cU')$ and $p_2(\cU')$ are non-empty open subsets of $\cZ$.
    But $\cZ$ is irreducible, 
    hence $\cU := p_1(\cU') \cap p_2(\cU')$ is non-empty and still open in $\cZ$.
    Let $V_0\in\cU$, then there is an open subset $\cU_0\subset \cZ$ such that
    $\Hom_A(V_0,\tau_A(W)) = \Hom_A(W,\tau_A(V_0)) = 0$ for all $W\in\cU_0$
    (in particular $W\not\cong V_0$ because $\dimhom^\tau_A(\cZ)=c_A(\cZ)\neq 0$).
    Let $V_1 \in \cU\cap \cU_0$,
    then there exists an open subset $\cU_1 \subseteq \cZ$ with 
    $\Hom_A(V_1,\tau_A(W)) = \Hom_A(W,\tau_A(V_1)) = 0$ for all $W\in\cU_1$.
    Let $V_2\in\cU \cap \cU_0\cap \cU_1$.
    Inductively, construct an infinite family $\cV := \{V_i \mid i\geq 0\}\subset \cZ$
    with $\Hom_A(V_i,\tau_A(V_j)) = \Hom_A(V_j,\tau_A(V_i)) = 0$ for all $i,j\geq 0$ with $i\neq j$.

    \bigskip

    Next, we want to define for each $i\geq 0$ 
    a $\theta$-stable factor $V_i\twoheadrightarrow X_i$
    with $X_i \in \cW(V_j)$ for all $j\neq i$.
    We then get $\Hom_A(X_i,X_j) = 0$ for all $i,j\geq 0$ with $i\neq j$
    because $X_i \in\fac(V_i) \subseteq {^\perp(V_i^\perp)} \subseteq {^\perp\cW(V_i)}$
    and $X_j\in\cW(V_i)$
    by construction.
    Note that the dimension of the $X_i$ for $i\geq 0$
    is bounded by the dimension of elements in $\cZ$.
    Thus, there is a dimension vector $\bd\in \K_0(A)^+$
    such that $X_i\in\Rep(A,\bd,\theta)^\st$ for infinitely many $i\geq 0$.
    Therefore, $\norm{\Mod(A,\bd,\theta)^\st} = \infty$,
    but $\Mod(A,\bd,\theta)^\st$ is a quasi-projective variety 
    hence its dimension must be at least $1$.

    \bigskip

    It remains to construct such stable factors:
    Fix $i\geq 0$ and take any $j\geq 0$ with $j\neq i$.
    Consider the canonical short exact sequence
    \begin{align*}
        0 \to U_j \to V_i \to W_j \to 0
    \end{align*}
    with $U_j \in {^\perp(V_j^\perp)}$ and $W_j\in V_j^\perp$.
    Since by choice, $V_i\in {^\perp(\tau_A(V_j))}$
    we do find $W_j\in\cW(V_j)$.
    Even more,
    we have $\Hom_A(W_j,\tau_A(V_i)) \neq 0$
    because $\Hom_A(V_i,\tau_A(V_i)) \neq 0$
    and $\Hom_A(U_j,\tau_A(V_i)) = 0$
    for $\tau_A(V_i) \in V_j^{\perp}$
    by choice.
    In particular, $W_j \neq 0$.

    \bigskip

    Let $k\geq 0$ with $k\neq i,j$.
    Consider the canonical short exact sequence
    \begin{align*}
        0 \to U_k \to W_j \to W_k \to 0
    \end{align*}
    with $U_k \in {^\perp(V_k^\perp)}$ and $W_k\in V_k^\perp$.
    As before,
    we find $W_k \in \cW(V_k)$.
    Furthermore,
    $W_k\in\fac(\cW(V_j))$
    but $\cW(V_j)$ is a Serre subcategory of $\cW_\theta$
    by Lemma \ref{lem:tau_perp_serre}
    and $W_k \in \cW(V_k) \subseteq \cW_\theta$
    hence we still have $W_k \in \cW(V_j)$.
    Also, 
    $\Hom_A(W_k,\tau_A(V_i)) \neq 0$
    because $\Hom_A(U_k,\tau_A(V_i)) = 0$
    and we know from before that $\Hom_A(W_j,\tau_A(V_i)) \neq 0$.

    \bigskip

    We can proceed like this for all indices different from $i$,
    to obtain a non-zero factor module $V_i\twoheadrightarrow X'_i$ 
    with $X'_i\in\cW(V_j)$ for all $j\neq i$.
    Finally take $X_i$ to be a $\theta$-stable factor of $X'_i$.
    Again, $X_i\in\cW(V_j)$ for all $j\neq i$
    because $\cW(V_j)$ is a Serre subcategory of $\cW_\theta$
    for all $j\geq 0$ by Lemma \ref{lem:tau_perp_serre}.
\end{proof}

\subsection*{Acknowledgments}
This work is part of my Ph.D. thesis.
I am deeply indebted to Syddansk Universitet (SDU)
as well as to my co-supervisors Prof. Christof Geiß and Dr. Fabian Haiden
for making this project possible.
Special thanks go to Prof. Christof Geiß 
for his invaluable continuous support
and for sharing his recent work with me.
I am also grateful to Dr. Fabian Haiden for his helpful advice 
and many stimulating ideas.
Last but not least, 
I would like to thank Dr. Hipolito Treffinger
for interesting discussions
and his constant encouragement
during my research visits at the Université Paris Cité.

This paper is partly a result of the 
ERC-SyG project, Recursive and Exact New Quantum Theory (ReNewQuantum) 
which received funding from the European Research Council (ERC) 
under the European Union's Horizon 2020 research and innovation programme 
under grant agreement No 810573,
held by Prof. Jørgen Ellegaard Andersen
to whom I would like to express my sincere gratitude 
for giving me the opportunity to pursue my Ph.D. studies at SDU.

\bibliographystyle{amsalpha}
\bibliography{bibliography.bib}

\end{document}